\documentclass [12pt,a4paper,reqno]{amsart}
\usepackage{amssymb}
\usepackage{epsfig}
\usepackage{verbatim}
\usepackage{pb-diagram}
\usepackage{color}
\usepackage{pgf}
\usepackage{tikz}
\usepackage{longtable}
\usepackage[a4paper,margin=1in,includehead,includefoot]{geometry}
\usepackage{nccmath}

\newcommand{\ef}{\end{equation}}
\chardef\bslash=`\\

\newtheorem{thm}{Theorem}[section]
\newtheorem*{thm*}{Theorem}

\newtheorem{lem}[thm]{Lemma}
\newtheorem{corl}[thm]{Corollary}

\newtheorem{prop}[thm]{Proposition}
\newtheorem{prop*}{Proposition}

\theoremstyle{definition}
\newtheorem{defn}[thm]{Definition}

\newtheorem{examp}[thm]{Example}
\newtheorem*{examp*}{Example}
\newtheorem*{remark*}{Remark}
\newtheorem*{Note*}{Note}
\newtheorem*{defn*}{Definition}

\theoremstyle{remark}
\newtheorem{remark}[thm]{Remark}

\newcommand\str{\raisebox{0pt}[13pt][6pt]{}}
\newcommand\mo{\negthinspace ^-\tmspace-{2mu}{.1667em}1}
\newcommand\mt{\negthinspace ^-\tmspace-{2mu}{.1667em}2}

\newlength\brone\newlength\brtwo\newlength\brthree
\newcommand\brac[1]{\settoheight{\brone}{\hbox{$#1$}}\addtolength\brone{-\brtwo}\left(\raisebox{-35pt}[\brone]{}\right.\hspace{-\brthree}#1\hspace{-\brthree}\left.\raisebox{0pt}[\brone]{}\right)}
\newcommand\ddts{\begin{picture}(8,8)\multiput(8,0)(-4,4)3{.}\end{picture}}

\newcommand\ppmod[1]{\ (\operatorname{mod}\,#1)}
\newcommand\bbf{\mathbb{F}}
\newcommand\bbc{\mathbb{C}}

\newcommand\amod{\operatorname{mod}}

\renewcommand{\sectionmark}[1]{}

\newcommand{\la}{\langle}
\newcommand{\ra}{\rangle}

\newcommand\bbz{\mathbb Z}
\newcommand\bbq{\mathbb Q}
\newcommand\bbr{\mathbb R}
\newcommand\re{\operatorname{re}}

\newcommand{\lm}{\lambda}
\newcommand{\Lm}{\Lambda}
\newcommand\fkg{\mathfrak{g}}
\newcommand\fkh{\mathfrak{h}}

\newcommand\dra\dashrightarrow
\newcommand\nod[4]{{\scriptsize $(#1,\negthinspace#2,\negthinspace#3)\rlap{$^{\negthinspace#4}$}$}}
\newcommand\nodhub[7]{{\scriptsize $\begin{matrix}(#1,\negthinspace#2,\negthinspace#3)\rlap{$^{\negthinspace#4}$}\\[-2pt][#5,\negthinspace#6,\negthinspace#7]\end{matrix}$}}
\newcommand\ow{\overset{\circ}W}

\newcommand{\defe} {\operatorname{def}}

\newcommand\oa{\overline\alpha_}

\renewcommand\leq\leqslant
\renewcommand\geq\geqslant
\renewcommand\le\leqslant
\renewcommand\ge\geqslant

\newlength\lc
\newlength\lo

\date{}
\begin{document}
\renewcommand{\bottomfraction}{0.999}
\renewcommand{\textfraction}{0.001}
\renewcommand\topfraction{0.999}


\tikzset{redx/.style={red,thick}}
\tikzset{greenx/.style={dashed,green!30!black,thick}}
\tikzset{bluex/.style={blue,dash pattern=on 1pt off 3pt on 3pt off 3pt,thick}}
\tikzset{purplex/.style={red!65!blue,thick}}
\tikzset{lbluex/.style={dashed,blue!40!black,thick}}
\tikzset{lgreenx/.style={green!40!black,dash pattern=on 1pt off 3pt on 3pt off 3pt,thick}}
\tikzset{greyx/.style={gray,thick}}
\newlength\ri
\newlength\lf
\newlength\rd
\newlength\ld
\newlength\dd
\newlength\ofs


\author{O.\ Barshevsky, M.\ Fayers and M.\ Schaps}

\title[A non-recursive criterion for weights] {A non-recursive criterion for weights of a highest-weight module for an affine Lie algebra}

\begin{abstract}
Let $\Lm$ be a dominant integral weight of level $k$ for the affine Lie algebra $\mathfrak g$
and let $\alpha$ be a non-negative integral combination of simple roots. We address the question of whether the weight $\eta=\Lm -\alpha$ lies in the set $P(\Lm)$ of weights in the irreducible highest-weight module with highest weight $\Lm$. We give a non-recursive criterion in terms of the coefficients of $\alpha$ modulo an integral lattice $kM$, where $M$ is the lattice parameterizing the abelian normal subgroup $T$ of the Weyl group. The criterion requires the preliminary computation of a set no larger than the fundamental region for $kM$, and we show how this set can be efficiently calculated.
\end{abstract}

\maketitle

\section {Introduction}

To an affine Lie algebra $\mathfrak g$ over $\mathbb C$ corresponds a vector space $\mathfrak h$ containing the Cartan subalgebra and a dual vector space $\mathfrak h^*$. An important class of modules are the \emph{integrable highest-weight modules} $L(\Lm)$ for $\Lm \in \mathfrak h^*$. We will study the set $P(\Lm)$ of weights of the homogeneous elements of $L(\Lm)$.

For each affine Lie algebra there is a certain integral lattice $M$ of weights of rank $\ell$ defined in \cite[6.5.8]{Ka}, parametrizing an abelian subgroup $T$ of the Weyl group. We determine a set $\bar N$ of maximal weights which is in bijection with the image of $P(\Lm)$ in a fundamental region for $kM$. We give a criterion for a weight $\eta$ to lie in $P(\Lm)$, which is a generalization of \cite[12.6.3]{Ka} to levels $k > 1$. In the process we also prove that every weight in $\Lm - Q$ can be $\delta$-shifted to an element of $P(\Lm)$ (where $Q$ denotes the root lattice and $\delta$ the null root). Finally, using a case-by-case study of the affine families and the affine exceptional algebras, we show how to separate the weights into $Q$-classes by congruences and how to find the maximal dominant weights in each class.

The original motivation for this research was an investigation of the existence of the block $H^\Lm_\alpha$ of the cyclotomic Hecke algebra $H^\Lm_d(\bbf,\xi)$ \cite{AK}, where $\xi \in \bbf^\times$ is a primitive $(\ell+1)$-th root of unity. This question is typically settled by a recursive construction of the weights of blocks up to rank $d$ or by the construction of a multipartition with content $\alpha$. By the categorification result in \cite{AM}, such a block exists if and only if the corresponding weight $\eta$ is in $P(\Lm)$ for the affine Lie algebra $\fkg(A^{(1)}_{\ell})$, so our non-recursive criterion gives a criterion in terms of the residues of the coefficients of $\alpha$ modulo $k$. In this case the set to be computed is of order $k^{\ell}$. A similar procedure involving $A^{(2)}_{2\ell}$ corresponds to cyclotomic Hecke algebras related to spin representations of the symmetric groups \cite{Kl}.

\subsection{Affine Lie algebras}

In this paper we work with the affine Lie algebra $\mathfrak g$ defined by an $(\ell+1)\times (\ell+1)$ Cartan matrix $A = [a_{ij}]$.
There are two families ($A^{(1)}_\ell$ and $D^{(1)}_\ell$) with symmetric Cartan matrices, several other families ($B^{(1)}_\ell$, $C^{(1)}_\ell$, $D^{(2)}_\ell$, $A^{(2)}_{2\ell-1}$ and $A^{(2)}_{2\ell}$) with non-symmetric matrices, and a number of exceptional algebras \cite[Chap.\ 4]{Ka}. The algebras with exponent $(1)$ will be called untwisted, and the algebras with exponent $(2)$ or $(3)$ will be called twisted.

The algebra $\mathfrak g$ has simple roots $\alpha_i$ and simple coroots $h_i$ for $i\in\{0,\dots,\ell\}$, with a pairing given by the Cartan matrix entries
\[\la h_i,\alpha_j\ra=a_{ij}.\]
We choose a set of fundamental weights $\{\Lm_i\mid i=0,\dots,\ell\}$ in $\fkh^\ast$ which satisfy $\la h_i,\Lm_j\ra=\delta_{ij}$. In the affine case, we have a \emph{null root}
\[
\delta=a_0\alpha_0+\dots+a_\ell\alpha_\ell
\]
with integer coefficients, which generates the kernel when $A$ acts from the left, and, dually, a canonical central element
\[
c = a_0^\vee h_0+\dots+a_\ell^\vee h_\ell,
\]
where the integral coefficients of $c$ are chosen so that that $\la c, \alpha_i\ra = 0$ for any simple root $\alpha_i$.

If we let $D$ be the matrix $\operatorname{diag}(a_0^{-1}a_0^\vee,\dots,a_\ell^{-1}a_\ell^\vee)$, then the matrix $B=DA$ is symmetric. This matrix determines an invariant symmetric bilinear form $(\cdot|\cdot)$, for which we have
\[(\alpha_i|\Lm_j)=\frac{a^\vee_i}{a_i}\delta_{ij},\qquad (\alpha_i|\alpha_j)=\frac{a^\vee_i}{a_i} a_{ij},\qquad (\alpha_i|\delta)=0,\qquad (\delta|\delta)=0.\]

Following Kac, we define
\begin{alignat*}3
P&=\{\eta \in \mathfrak{h}^*&\mid\la h_i,\eta\ra &\in \mathbb{Z}\text{ for }&&i=0,\dots,\ell\},\\
P_+&=\{\eta \in P&\mid\la h_i,\eta\ra &\geq 0\text{ for }&&i=0,\dots,\ell\},\\
Q&=\textstyle\sum_{i=0}^\ell\rlap{$\mathbb Z\alpha_i$,}&&&&\\
Q_+ &=\textstyle\sum_{i=0}^\ell\rlap{$\mathbb Z_{\geq0}\alpha_i$.}&&&&
\end{alignat*}
The weights in $P$ are called \emph{integral weights}, and those in $P_+$ are \emph{dominant integral weights}.

\begin{defn} The \emph{level} of a dominant integral weight $\Lm$ is the integer $k=\la c,\Lm\ra$, or equivalently $(\Lm|\delta)$.
\end{defn}

\subsection{The weights of an integrable highest-weight module}

In this paper we fix a dominant integral weight $\Lm\in P_+$. We let $P(\Lm)$ denote the set of weights labeling non-zero weight spaces in the integrable highest-weight module $L(\Lm)$. Then $P(\Lm)$ is a subset of $\Lm-Q_+$.

Although the set $P(\Lm)$ is discussed in detail in \cite{Ka}, it is not easy to determine whether a given weight lies in $P(\Lm)$. The purpose of this paper is to give a non-recursive way to do this.

A weight $\eta\in P(\Lambda)$ is called \textit{maximal} if $\eta+\delta$ is not a weight for $\Lambda$, and the set of maximal weights is denoted $\max(\Lambda)$.
By \cite[(12.6.1)]{Ka} the set of weights for $L(\Lambda)$ is the union of the negative shifts by the null root $\delta$ of the maximal weights:
$$P(\Lambda) = \left\{\lambda-s\delta \ \middle|\ \lambda \in \max(\Lambda), s \in \mathbb{Z}_{\geq0}\right\}.$$

\subsection{The Weyl group}\label{weylsec}

Let $W$ denote the Weyl group of $\fkg$; considered as a group acting on the weight space $\fkh^\ast$, this is generated by the reflections $s_0,\dots,s_\ell$ defined by
\[
s_i:\eta\longmapsto \eta-\la h_i,\eta\ra\alpha_i.
\]
The Weyl group is of critical importance to this paper, since it fixes the set $P(\Lm)$ of weights for $L(\Lm)$.  It also fixes the null root $\delta$, and therefore fixes the set $\max(\Lm)$.

The criterion we shall prove in the next section depends on the decomposition of $W$ as a semi-direct product $T\rtimes\overset\circ W$, where $\overset\circ W$ is the (finite) subgroup generated by $s_1,\dots,s_\ell$, and $T$ is a torsion-free abelian group.  The elements of $T$ correspond to a certain lattice $M$ in $\fkh^\ast$, which has generators
\[\frac{d_1\alpha_1}{a_0},\dots,\frac{d_\ell\alpha_\ell}{a_0}\]
for certain positive integers $d_1,\dots,d_\ell$, which are given in \cite[(6.5.8)]{Ka}.  When the Cartan matrix is symmetric (types $A^{(1)}_\ell$ and $D^{(1)}_\ell$) or of twisted type other than $A^{(2)}_{2\ell}$, all the $d_i$ are equal to $1$.  For $\alpha\in M$, the corresponding element $t_\alpha$ of $T$ is given by \cite[(6,5,2)]{Ka}:
\[
t_\alpha:\eta\mapsto\eta+k\alpha-\left((\eta|\alpha)+\tfrac12(\alpha|\alpha)k\right)\delta,
\]
where $k=\la c,\eta\ra$ is the level of $\eta$.

\section{The criterion for membership in $P(\Lm)$}

In this section, we give a theorem which yields an algorithm for determining whether a given weight $\eta\in\Lm-Q_+$ lies in $P(\Lm)$.  Writing $\eta=\Lm-\sum_ib_i\alpha_i$, we may refer to $\eta$ by its \emph{content} $b=(b_0,\dots,b_\ell)$, if $\Lm$ is understood. We will give another representation of $\eta$ in terms of the decomposition $W=T\rtimes\overset \circ W$.

By \cite[Corollary 10.1]{Ka}, a weight $\eta\in P(\Lm)$ is equivalent under the action of the affine Weyl group $W$ to a unique element of $P_+\cap P(\Lm)$. From this we can obtain a description of the set $P(\Lm)$ modulo $\delta$. First we need to know that every weight of positive level is $W$-equivalent to a dominant weight.

\begin{prop}\label{tits}
Suppose $\eta\in P$ with $\la c,\eta\ra>0$. Then there is $w\in W$ such that $w\eta\in P_+$.
\end{prop}

\begin{proof}
This is essentially the result of \cite[Proposition 5.8(a)]{Ka}, though one must interchange $\fkg$ with its dual. What Kac proves in [\textit{loc.\ cit.}] is that the Tits cone
\[
W\cdot\left\{h\in \fkh_\bbr\ \middle|\ \la h,\alpha_i\ra\geq0\text{ for all }i\right\}
\]
includes all elements $h\in\fkh_\bbr$ for which $\la h,\delta\ra>0$. Applying this result to the algebra ${}^t\fkg$ dual to $\fkg$ (that is, the Kac--Moody algebra whose Cartan matrix is the transpose of the Cartan matrix of $\fkg$), and then interchanging $h_i$ and $\alpha_i$ for each $i$, one obtains that for $\fkg$ the cone
\[
W\cdot\left\{\lm\in\fkh^\ast_\bbr\ \middle|\ \la h_i,\lm\ra\geq0\text{ for all }i\right\}
\]
includes all elements $\eta\in\fkh^\ast_\bbr$ for which $\la c,\eta\ra>0$.
\end{proof}

Now we can show that, modulo $\delta$, $P(\Lm)$ coincides with $\Lm-Q$.

\begin{prop}\label{inclusive}
Suppose $\eta\in\Lm-Q$. Then there is some $s\in\bbz$ such that $\eta+s\delta\in P(\Lm)$.
\end{prop}

\begin{proof}
In the case where $\eta$ is a dominant weight, this follows from \cite[Proposition 11.2]{Ka}, since then every weight $\eta+s\delta$ is dominant, and for sufficiently small $s$ we must have $\eta+s\delta\leq\Lm$.

In general, we have $\la c,\eta\ra=\la c,\Lm\ra=k>0$, and so by Proposition \ref{tits}, $\eta$ is the image of some dominant weight $\xi$ under the action of the Weyl group. Since the Weyl group action involves adding elements of $Q$, $\xi$ also lies in $\Lm-Q$, and so the present proposition holds for $\xi$. Since $\delta$ and $P(\Lm)$ are fixed by the Weyl group action, the result holds for $\eta$ too.
\end{proof}

\begin{defn}\label{delta-shift}
Fix a dominant integral weight $\Lm$. By Proposition \ref{inclusive}, for every weight $\eta \in \Lm -Q$ there exists an integer $s$ such that $\eta +s\delta \in P(\Lm)$. The largest such integer will be denoted by $s(\eta)$ and will be called the \emph{$\delta$-shift} of $\eta$.
\end{defn}

The basis for the membership criterion is a calculation of this $\delta$-shift, combined with the following lemma.

\begin{lem}\label{moddelta}\indent
\begin{enumerate}
\vspace{-\topsep}
\item If $\eta = \Lm-\alpha$ for $\alpha \in Q$, then the $\delta$-shift $s(\eta)$ is well-defined and $\eta +s(\eta)\delta$ is a maximal weight.
\item The weight $\eta$ lies in $P(\Lm)$ if and only if $s(\eta) \geq 0$. In this case, $\alpha \in Q_+$.
\end{enumerate}
\end{lem}

\begin{proof}\indent
\begin{enumerate}
\vspace{-\topsep}
\item
By Proposition \ref{inclusive}, $\eta+s\delta \in P(\Lm)$, for some $s$. We must have $\alpha -s\delta \in Q_+$, so $s$ is bounded above by the minimal coefficients in $\alpha$, which shows that $s(\eta)$, the maximum of all the $s$, must exist. If $\zeta = \eta+s(\eta)\delta$ is not a maximal weight, then $\zeta +\delta= \eta+(s(\eta)+1)\delta \in P(\Lm)$, in contradiction to the maximality of $s(\eta)$.

\item We have $\eta =\zeta-s(\eta)\delta$. As noted above, the weights in $P(\Lm)$ are precisely the weights $\zeta-s\delta$ for $s \geq 0, \zeta \in \max(\Lm)$. Since
$\zeta = \Lm - \alpha'$ for $\alpha' \in Q_+$, we surely have $\alpha = \alpha'+s(\eta)\delta \in Q_+$.\qedhere
\end{enumerate}
\end{proof}

Now recall from \S\ref{weylsec} the decomposition $W=T \rtimes \overset \circ W$, and the integers $d_1,\dots,d_\ell$ defining the lattice $M$.

\begin{defn}\indent
Suppose $\Lm$ is of level $k$.
\begin{itemize}
\item
For any $b=(b_0,\dots,b_\ell) \in \mathbb Z^{\ell+1}$, not necessarily non-negative, let $\eta(b)$ be the weight $\Lm-(b_0\alpha_0+\dots+b_\ell\alpha_\ell)$.
\item
For any $\eta = \eta(b)$, let
$$\tilde \eta= \big((a_0b_1-a_1b_0) \amod kd_1,\dots, (a_0b_\ell-a_\ell b_0)\amod kd_\ell\big) \in \prod_{i=1}^\ell (\mathbb Z/kd_i \mathbb Z).$$
Note that if $\eta,\zeta\in\Lm-Q$, with $\zeta-\eta=\sum_{i=0}^\ell c_i\alpha_i$, then we have $(\tilde\zeta-\tilde\eta)_i\equiv a_ic_0-a_0c_i\amod kd_i$.
\end{itemize}
\end{defn}

Let $N$ be the set of maximal dominant weights, and let $\tilde N=\overset \circ W \cdot N$ be the union of the orbits under the finite Weyl group $\overset \circ W$. Say that two elements of $\tilde N$ lie in the same \emph{$T$-class} if one can be moved to the other by the action of an element of $T$, and let $\bar N$ be a set of representatives of the $T$-classes of $\tilde N$.

\begin{prop}\label{zeta}
Suppose $\Lm$ is a a dominant integral weight of level $k$.
\begin{enumerate}
\item
If $\eta,\zeta\in\Lm-Q$, then we have $\eta = t_\alpha(\zeta)-s\delta$ for some $\alpha \in M$ and some $s\in\bbz$ if and only if $\tilde \eta = \tilde \zeta$.
\item
If $f:\bar N \rightarrow\prod_{i=1}^\ell (\mathbb Z/kd_i \mathbb Z)$ denotes the restriction of $\tilde \cdot$, then $f$ is injective.
\item
If $\fkg$ is of any type other than $A^{(2)}_{2\ell}$, then $f$ is surjective, and thus $\bar N$ has $k^\ell \prod_{i=1}^\ell d_i$ elements.
\item
If $\fkg=\fkg(A^{(2)}_{2\ell})$, then the image of $f$ is $(2\bbz/2k\bbz)^{\ell-1}\times\bbz/k\bbz$, and thus $\bar N$ has $k^{\ell}$ elements.
\end{enumerate}
\end{prop}
\begin{proof}\indent
\begin{enumerate}
\vspace{-\topsep}
\item First assume that $\eta = t_\alpha(\zeta)-s\delta$ for some $\alpha\in M$, and write
\[
\alpha=\sum_{i=1}^\ell \mfrac{n_id_i}{a_0}\alpha_i
\]
for integers $n_1,\dots,n_\ell$. From above, we have
\[
\eta = t_\alpha(\zeta)-s\delta=\zeta +k\alpha-\left((\zeta|\alpha)+\tfrac{1}{2}(\alpha|\alpha)k+s\right)\delta.
\]
Writing the coefficient $(\zeta|\alpha)+\tfrac{1}{2}(\alpha|\alpha)k+s$ as $N$, we get
\begin{align*}
\zeta-\eta&=N\delta-k\alpha\\
&=N\sum_{i=0}^\ell a_i\alpha_i-k\sum_{i=1}^\ell\mfrac{n_id_i}{a_0}\alpha_i;
\end{align*}
hence the $i$th component of $\tilde\zeta-\tilde\eta$ is
\[
\left(a_iNa_0-a_0\left(Na_i-k\mfrac{n_id_i}{a_0}\right)\right)\amod kd_i,
\]
which is zero.

Conversely, suppose that $\tilde \eta=\tilde\zeta$ with $\eta = \eta(b)$ and $\zeta=\eta(b')$; then we can write
$$(a_0b_1'-a_1 b_0',\dots,a_0b_\ell'-a_\ell b_0')-(a_0b_1-a_1b_0,\dots,a_0b_\ell-a_\ell b_0)=k(n_1d_1,\dots,n_\ell d_\ell)$$
for some $n_1,\dots,n_\ell\in\bbz$. Thus, if we set $\alpha=\mfrac{n_1d_1}{a_0}\alpha_1+\dots+\mfrac{n_\ell d_\ell}{a_0}\alpha_\ell$, then modulo $\bbq\delta$ we have
\begin{align*}
t_\alpha(\zeta)&\equiv \zeta+k\alpha\\
&\equiv \Lm-\sum_{i=0}^\ell b_i'\alpha_i+k\sum_{i=1}^\ell\mfrac{n_id_i}{a_0}\alpha_i\\
&\equiv \Lm-(b_0+(b'_0-b_0))\alpha_0-\sum_{i=1}^\ell\left(b_i'-\mfrac{a_0b_i'-a_ib_0'-a_0b_i+a_ib_0}{a_0}\right)\alpha_i\\
&\equiv \Lm-\sum_{i=0}^\ell b_i\alpha_i-\mfrac{b_0'-b_0}{a_0}\sum_{i=0}^\ell a_i\alpha_i\\
&\equiv \eta-\mfrac{b_0'-b_0}{a_0}\delta\\
&\equiv \eta.
\end{align*}
So $\eta = t_\alpha(\zeta)+s\delta$ for some rational number $s$; but since $\eta-t_\alpha(\zeta)\in Q$ and $Q\cap\bbq\delta=\bbz\delta$, $s$ is in fact an integer.
\item
If $\eta$ and $\zeta$ are in $\bar N$ with $f(\eta)=f(\zeta)$, then from (1) we have $\eta = t_\alpha(\zeta)+s\delta$ for some $s$. But since $\eta,\zeta$ lie in $\bar N$, they are both maximal weights and thus $s=0$, so they are both in the same $T$-class. Since $\bar N$ contains a unique representative of each $T$-class, $\eta=\zeta$ and thus $f$ is injective.
\item
Since $\fkg$ is not of type $A^{(2)}_{2\ell}$, we have $a_0=1$.  Given an element $(b_1,\dots,b_\ell) \in \prod_{i=1}^\ell (\mathbb Z/kd_i \mathbb Z)$, consider the content $b=(0,b_1,\dots,b_\ell)$ (regarding each $b_i$ as an integer in the range $\{0,\dots,kd_i-1\}$). By Lemma \ref{inclusive}, the corresponding weight $\eta(b)$ can be written as $t'w_0(\zeta')-s\delta$ for some $\zeta' \in N$. If $\zeta$ is the representative of the $T$-class of $w_0\zeta'$ in $\bar N$, then we replace $t'$ by another element $t \in T$ such that $t'w_0(\zeta')=t(\zeta)$, giving $\eta=t(\zeta)-s\delta$. By part (1) of this lemma, we have
$f(\zeta)=\tilde \zeta = \tilde \eta(b)=(b_1,\dots,b_\ell)$, so $f$ is indeed surjective.

The cardinality of $\bar N$ is now immediate, since the number of elements in $\prod_{i=1}^\ell(\mathbb Z/kd_i\mathbb Z)$ is $k^\ell\prod_{i=1}^\ell d_i$.
\item
In this case, we have $a_i=d_i=2$ for $i=1,\dots,\ell-1$, while $a_\ell=d_\ell=1$.  If $\eta=\eta(b)$, then for $i<\ell$ we have $\tilde\eta_i=2b_i-2b_0$, which is even; so the image of $f$ is certainly contained in $(2\bbz/2k\bbz)^{\ell-1}\times\bbz/k\bbz$.  Conversely, suppose we are given integers $c_1,\dots,c_\ell$ with $c_1,\dots,c_{\ell-1}$ even.  If $c_\ell$ is also even, then let $b$ be the content $\left(0,\mfrac{c_1}2,\dots,\mfrac{c_\ell}2\right)$ and let $\eta=\eta(b)$.  Repeating the argument from above, there is $\zeta\in\bar N$ with $f(\zeta)=\tilde\eta=(c_1,\dots,c_\ell)$.  If instead $c_\ell$ is odd, let $b=\left(1,\mfrac{c_1}2+1,\dots,\mfrac{c_{\ell-1}}2+1,\mfrac{c_\ell+1}2\right)$ and repeat the argument.\qedhere
\end{enumerate}
\end{proof}

Now we can give our main result.

\begin{thm}\label{main}
Suppose $\mathfrak g$ is an affine Lie algebra, with Weyl group $W=T \rtimes \overset \circ W$. Let $\Lambda$ be a dominant integral weight, $N$ the set of maximal dominant weights in $P(\Lm)$, and $\bar N$ a set of $T$-class representatives of the set $\overset \circ W \cdot N$.
For $\eta = \eta(b)$, let $\zeta$ be the unique element from $\bar N$ such that $\tilde\eta=\tilde\zeta$. Let $\alpha \in M$ and $s\in\bbz$ be determined by the formula
$\eta=t_\alpha(\zeta)-s(\eta)\delta$.
\begin{enumerate}
\item
Writing $\eta -\zeta=c_0\alpha_0+\dots+c_\ell \alpha_\ell$, we have
$$\alpha=\mfrac{1}{ka_0}\big((a_0c_1-a_1c_0)\alpha_1+\dots+(a_0c_\ell-a_\ell c_0)\alpha_\ell\big),$$
from which we can calculate
\[s(\eta)=-\mfrac{c_0}{a_0}-\left((\zeta|\alpha)+\tfrac{1}{2}(\alpha|\alpha)k\right).\]
\item
$\eta$ lies in $P(\Lm)$ if and only if $s(\eta)\geq0$.
\end{enumerate}
\end{thm}

\begin{proof}\indent
\begin{enumerate}
\vspace{-\topsep}
\item
By Proposition \ref{zeta}, there is a unique $\zeta \in \bar N$ such that $\tilde \eta = \tilde \zeta$. There is an $\alpha \in M$ and an $s \in \bbz$ such that $\eta = t_\alpha(\zeta)-s\delta.$ However, since $\zeta \in \max(\Lm)$, we get $s=s(\eta)$ by Definition \ref{delta-shift}.

Now we have
\begin{align*}
\sum_{i=0}^\ell c_i\alpha_i&=\eta -\zeta\\
&= t_\alpha(\zeta) -s(\eta)\delta -\zeta\\
&=k\alpha -\left((\zeta|\alpha)+\tfrac{1}{2}(\alpha|\alpha)k+s(\eta)\right)\delta.
\end{align*}
Since $\alpha$ lies in the span of $\alpha_1,\dots,\alpha_\ell$, we can compare coefficients of $\alpha_0$ to get
\[
-\left((\zeta|\alpha)+\tfrac{1}{2}(\alpha|\alpha)k+s(\eta)\right)a_0=c_0,
\]
from which we get
\begin{align*}
\alpha&=\mfrac1k\sum_{i=0}^\ell c_i\alpha_i-\mfrac{c_0}{a_0}\delta\\
&=\mfrac1{ka_0}\sum_{i=1}^\ell(a_0c_i-a_ic_0)\alpha_i
\end{align*}
and also
\[
s(\eta)=-\mfrac{c_0}{a_0}-\left((\zeta|\alpha)+\tfrac{1}{2}(\alpha|\alpha)k\right).
\]
\item Since $t(\zeta)$ lies in $\max(\Lm)$, we have $\eta \in P(\Lm)$ if and only if $s(\eta)$ is non-negative, by Lemma \ref{moddelta}.\qedhere
\end{enumerate}
\end{proof}

\begin{defn}\label{defect} \cite[3.11]{BK} The \textit{defect} of the weight $\eta=\Lambda-\alpha$ is given by
\[
\defe(\eta) = (\Lambda|\alpha)-\tfrac{1}{2}(\alpha|\alpha).
\]
\end{defn}

The defect is always non-negative, as was shown for a general Kac-Moody algebra in  \cite[Proposition 9.12]{Ka}

\begin{examp}

The weight $\Lm =\Lm_0 +\Lm_1$ for the twisted affine algebra $\fkg(A^{(2)}_4)$ is of level $3$. We represent the elements of $P(\Lm)$ as nodes in a graph, with coordinates corresponding to content. We wish to find the weights $\zeta =\eta(d) \in \bar N$.

Using the results of the next section, we can check that there are three maximal dominant weights in $P(\Lm)$, which we label by their contents as follows:
\[
a^0=(0,0,0),\qquad a^1=(1,1,0),\qquad a^2=(1,2,1).
\]
To find $\tilde N$, we apply the generators $s_1,s_2$ of $\overset\circ W$ to these three weights.  Each of $a_0,a_1,a_2$ yields a different orbit; these orbits have sizes $4,4,1$ respectively. 
Taking the nine elements $\zeta\in\tilde N$ and computing the vectors $\tilde\zeta$, we obtain a complete set of representatives for $(\mathbb Z/3\mathbb Z)^2$.
The weights $\eta(b)\in L(\Lm)$ with $b_0 \leq 2$ are depicted in Figure \ref{spin}, which is drawn as the projection of a three-dimensional model; for clarity, the only vertical lines drawn are those which would be visible in an opaque model, and only the maximal weights are labelled. The superscript on the content of any weight $\eta$ is its defect. Note that the defects are preserved by the action of the Weyl group, which acts by reflection on lines.

In Section \ref{s:examp}, we treat another example (in type $A^{(1)}_2$) in more detail.
\begin{figure}[htb]\label{spin}
\begin{center}
\setlength\ri{3.2cm}
\setlength\lf{2.2cm}
\setlength\rd{1.2cm}
\setlength\ld{1.2cm}
\setlength\dd{1.2cm}
\setlength\ofs{1mm}
\pagebreak

\[
\begin{tikzpicture}[scale=1.005]
\draw[lgreenx](-\lf,-2\dd-\ld)--++(-\lf,-\ld)--++(2\ri,-2\rd)--++(-3\lf,-3\ld)--++(\ri,-\rd)--++(\lf,\ld)--++(-2\ri,2\rd)--++(3\lf,3\ld)--cycle;
\draw[lgreenx](-3\lf+\ri,-2\dd-3\ld-\rd)--++(\ri,-\rd);
\draw[lbluex](0,-\dd)--++(-2\lf,-2\ld)--++(2\ri,-2\rd)--++(-2\lf,-2\ld);
\draw[lbluex](-\lf,-\dd-\ld)--++(\ri,-\rd)--++(-2\lf,-2\ld)--++(\ri,-\rd);
\foreach \x in {0,1,2}
\draw[greyx](2\ri-2\lf-\x\lf,-\dd-2\rd-2\ld-\x\ld)--++(0,-\dd);
\foreach \x in {0,1}
\draw[greyx](-2\lf+\x\ri-\x\lf,-\dd-\x\rd-2\ld-\x\ld)--++(0,-\dd);
\draw[greyx](-\lf+\ri,-\dd-\rd-\ld)--++(0,-\dd);
\draw[purplex](0,0)--++(-\lf,-\ld)--++(\ri,-\rd)--++(-\lf,-\ld);
\foreach \x in {0,1}
\foreach \y in {0,1}
\draw[greyx](-\x\lf-\y\lf+\y\ri,-\x\ld-\y\ld-\y\rd)--++(0,-\dd);
\draw(-\ofs-\lf,-2\dd-\ld)node[lgreenx,fill=white,inner sep=0pt]{\nod2100};
\draw(-2\lf,-2\dd-2\ld)node[lgreenx,fill=white,inner sep=0pt]{\nod2200};
\draw(-\lf+\ri,-2\dd-\ld-\rd)node[lgreenx,fill=white,inner sep=0pt]{\nod2110};
\draw(-4\lf+\ri,-2\dd-4\ld-\rd)node[lgreenx,fill=white,inner sep=0pt]{\nod2410};
\draw(-2\lf+2\ri,-2\dd-2\ld-2\rd)node[lgreenx,fill=white,inner sep=0pt]{\nod2220};
\draw(-5\lf+2\ri,-2\dd-5\ld-2\rd)node[lgreenx,fill=white,inner sep=0pt]{\nod2520};
\draw(-4\lf+3\ri,-2\dd-4\ld-3\rd)node[lgreenx,fill=white,inner sep=0pt]{\nod2430};
\draw(-5\lf+3\ri,-2\dd-5\ld-3\rd)node[lgreenx,fill=white,inner sep=0pt]{\nod2530};
\draw(0,-\dd)node[lbluex,fill=white,inner sep=0pt]{\nod1000};
\draw(-\lf,-\dd-\ld)node[lbluex,fill=white,inner sep=0pt]{\nod1101};
\draw(-2\lf,-\dd-2\ld)node[lbluex,fill=white,inner sep=0pt]{\nod1200};
\draw(-\lf+\ri,-\dd-\ld-\rd)node[lbluex,fill=white,inner sep=0pt]{\nod1111};
\draw(-2\lf+\ri,-\dd-2\ld-\rd)node[lbluex,fill=white,inner sep=0pt]{\nod1212};
\draw(-3\lf+\ri,-\dd-3\ld-\rd)node[lbluex,fill=white,inner sep=0pt]{\nod1311};
\draw(-2\lf+2\ri,-\dd-2\ld-2\rd)node[lbluex,fill=white,inner sep=0pt]{\nod1220};
\draw(-3\lf+2\ri,-\dd-3\ld-2\rd)node[lbluex,fill=white,inner sep=0pt]{\nod1321};
\draw(-4\lf+2\ri,-\dd-4\ld-2\rd)node[lbluex,fill=white,inner sep=0pt]{\nod1420};
\draw(0,0)node[purplex,fill=white,inner sep=0pt]{\nod0000};
\draw(-\lf,-\ld)node[purplex,fill=white,inner sep=0pt]{\nod0100};
\draw(-\lf+\ri,-\ld-\rd)node[purplex,fill=white,inner sep=0pt]{\nod0110};
\draw(\ofs-2\lf+\ri,-2\ld-\rd)node[purplex,fill=white,inner sep=0pt]{\nod0210};
\end{tikzpicture}
\]

\caption{}
\end{center}
\end{figure}
\end{examp}

\section{Finding the dominant weights in $P(\Lm)$ for affine Lie algebras}

In this section we describe how to find the dominant weights in the weight space $L(\Lm)$, for a fixed dominant integral weight $\Lm$.

\begin{defn}
We say that integral weights $\Lm$ and $\Lm'$ for $\fkg$ are \emph{equivalent}, written $\Lm \equiv \Lm'$, if $\Lm-\Lm' \in Q$.
\end{defn}

Since each element of $Q$ has level $0$, two equivalent weights must have the same level.  Moreover, since the set of weights of level $0$ is precisely the $\bbc$-span of $Q$, it is clear that two weights $\Lm,\Lm'$ of the same level differ by a linear combination of simple roots.  However, we do not have $\Lm'-\Lm \in Q$ unless all the coefficients are integers. We shall study this equivalence relation among integral weights of the same level below, treating each affine family separately.

\subsection{Positive hubs}

\begin{defn}
Suppose $\eta\in\fkh^\ast$.  The \emph{hub} of $\eta$ is the $(\ell+1)$-tuple $\theta(\eta)=(\theta_0,\dots,\theta_{\ell})$ defined by $\theta_i=\la h_i,\eta\ra$ for each $i$. We say that the hub of $\eta$ is \emph{positive} if each $\theta_i$ is non-negative.
\end{defn}

\begin{remark}\indent
\begin{enumerate}
\vspace{-\topsep}
\item
Note that $\theta(\eta)$ is just the projection of $\eta$ onto the first $\ell+1$ components in its representation with respect to the basis
\[B_H = \{\Lm_0,\Lm_1,\dots, \Lm_\ell, \delta\}\]
for $\fkh^\ast$. As such, it determines $\eta$ up to addition of a multiple of $\delta$; hence if $\theta$ is the hub of a weight in $P(\Lm)$, there will be a unique maximal weight in $P(\Lm)$ with hub~$\theta$.
\item
In an earlier work on cyclotomic Hecke algebras by the second author \cite{Fa}, the term ``hub'' was used for the negative of the hub defined here, and then the hubs of interest were the ``negative hubs''. We have reversed the sign here to make our work compatible with the conventional notation in affine Lie algebras.
\end{enumerate}
\end{remark}

\subsection{Finding the maximal dominant weights in $P(\Lm)$}

If we can find the set of all dominant weights $\Lm'$ equivalent to  $\Lm$, it is straightforward to find the maximal dominant weight for each one, using the following result.

\begin{prop}\label{maxcont}

Suppose $\eta$ is a dominant weight in $P(\Lm)$, and write $\eta=\Lm-\sum_i\gamma_i\alpha_i$.  If $\eta$ is maximal, then we have $\gamma_i<a_i$ for some $i$.
\end{prop}

\begin{proof}
For irreducible affine Kac--Moody algebras, \cite[Proposition 11.2]{Ka} tells us that the dominant weights in $P(\Lm)$ are precisely the dominant weights $\eta$ such that $\eta\leq\Lm$. If $\eta$ is a dominant weight in $\max(\Lm)$ with $\gamma_i\geq a_i$ for each $i$, then we have $\eta+\delta\leq\Lm$.  $\eta+\delta$ is a dominant weight (it has the same hub as $\eta$), and therefore must lie in $P(\Lm)$, contradicting the maximality of $\eta$.
\end{proof}

Using this proposition, we can find the maximal dominant weights in $P(\Lm)$ simply by finding which positive hubs occur as hubs of elements of $\Lm-Q$: given such a hub $\kappa$, there is a weight $\Lm'\in P(\Lm)$ with hub $\kappa$.  The unique maximal weight in $P(\Lm)$ with hub $\kappa$ is obtained by adding a multiple of $\delta$ to $\Lm'$, and Proposition \ref{maxcont} tells us what this multiple of $\delta$ must be, since there is a unique multiple of $\delta$ which will make each $\gamma_i$ non-negative with $\gamma_i<a_i$ for some $i$.

\begin{examp}
Suppose we are in type $D^{(1)}_5$, and $\Lm=\Lm_2$, which is of level $2$.  Using the results of the next section for this type, a positive  hub $\kappa=(\kappa_0,\dots,\kappa_5)$ is the hub of a weight in $\Lm-Q$ if and only if
\[
\kappa_0+\kappa_1+2\kappa_2+2\kappa_3+\kappa_4+\kappa_5=2
\]
and
\[\kappa_0-\kappa_1+2\kappa_2+2\kappa_4\in2+4\bbz.\]
It is easy to check that the positive hubs satisfying these criteria are
\[(0,0,1,0,0,0),\ \ (2,0,0,0,0,0),\ \ (0,2,0,0,0,0),\ \ (0,0,0,0,1,1).\]
The corresponding weights in $P(\Lm)$ are
\begin{align*}
&\Lm_2,\\
&\Lm_2-\alpha_1-2\alpha_2-2\alpha_3-\alpha_4-\alpha_5,\\
&\Lm_2-\alpha_0-2\alpha_2-2\alpha_3-\alpha_4-\alpha_5,\\
&\Lm_2-\alpha_0-\alpha_1-2\alpha_2-\alpha_3.
\end{align*}
\end{examp}

Now we show how to find all the maximal dominant elements of $P(\Lm)$.  We do this by first finding the hubs of these elements, and then considering the possible contents of these weights.

Since $\delta$ is an integral combination of the roots $\alpha_i$, the equivalence relation $\equiv$ on integral weights descends to an equivalence relation on hubs: given hubs $\theta,\kappa$, we can write $\theta\equiv\kappa$ if $\theta-\kappa\in\sum_i\mathbb Z\oa i$, where $\oa i$ denotes the hub of $\alpha_i$.  Our purpose is here is to give a simple description of this equivalence relation in terms of the entries of these hubs. It turns out that this can be done easily using a few simple congruences relating the coefficients of the hub, mostly modulo $2$.

In the following proposition, we use the labeling of simple roots from \cite[pp.\ 53--55]{Ka}; note that Cartan matrices have rows and columns numbered from $0$ to $\ell$.

\begin{prop}\label{smallpath}
Suppose $\fkg$ is an affine Lie algebra, and $\theta,\kappa$ are hubs of the same level. Let $\psi = \theta -\kappa$. Then $\psi \in\sum_i\oa i$ if and only if the coordinates of $\psi$ satisfy the congruences given in the following table.

\newlength\lenn
\setlength\lenn{312pt}
\begin{center}
\textup{\begin{longtable}{|c|c|}\hline
 Type&Congruences\\\endhead\hline
\str$A_\ell^{(1)} (\ell\geq1)$&
$\psi_1+2\psi_2+\dots+\ell\psi_\ell\equiv 0\pmod{\ell+1}$\\\hline
\str$B_\ell^{(1)}$ ($\ell\geq3$)&
$\psi_\ell\equiv\    0\pmod2$\\\hline
\str$C^{(1)}_\ell$ ($\ell \geq 2$)&
$\psi_0+\psi_2+\psi_4+\dots\equiv 0\pmod2$\\\hline
\raisebox{0pt}[21pt][14pt]{}$D^{(1)}_\ell$ ($\ell \geq 4$, $\ell$ even)&
\begin{tabular}{c}$\psi_0+\psi_1\equiv 0\pmod2$ and\\
$\psi_1+\psi_3+\psi_5+\dots+\psi_{\ell-1}\equiv 0\pmod2$\end{tabular}\\\hline
\str$D^{(1)}_\ell$ ($\ell \geq 5$, $\ell$ odd)&
$\psi_0-\psi_1+2\psi_2+2\psi_4+2\psi_6+\dots+2\psi_{\ell-1}\equiv0\pmod4$\\\hline
\str$D^{(2)}_{\ell+1}$ ($\ell \geq 2$)&
$\psi_0\equiv 0\pmod2$\\\hline
\str$A^{(2)}_{2\ell-1}$ ($\ell \geq 3$)&
$\psi_1+\psi_3+\psi_5+\dots\equiv 0\pmod2$\\\hline
\str$A^{(2)}_{2\ell}$ ($\ell \geq 1$)&
None\\\hline
\str$E^{(1)}_{6}$&
$\psi_0+2\psi_6\equiv \psi_5+2\psi_4\pmod3$\\\hline
\str$E^{(1)}_{7}$&
$\psi_0+\psi_2 +\psi_7\equiv 0 \pmod2$\\\hline
\str$E^{(1)}_{8}$&
None\\\hline
\str$E^{(2)}_{6}$&
None\\\hline
\str$F^{(1)}_{4}$&
None\\\hline
\str$G^{(1)}_{2}$&
None\\\hline
\str$D^{(3)}_{4}$&
None\\\hline
\end{longtable}}
\end{center}
\end{prop}

\begin{proof}
For brevity we will denote $\sum_i\bbz\oa i$ by $\bar Q$.  The hypothesis that $\theta$ and $\kappa$ have the same level implies that $\psi$ has level $0$.  We wish to show that this fact and the given additional congruences are equivalent to the statement that $\psi \in\bar Q$.  That an element of $\bar Q$ must satisfy the given congruences follows from a straightforward inspection of the Cartan matrix; conversely, given $\psi$ satisfying the required congruences, our technique will be to subtract integer multiples of various $\oa i$ until we are left with a hub (necessarily also of level $0$ and also satisfying the required congruences) in which most of the coordinates are zero.  It will then be straightforward to see that the latter hub lies in $\bar Q$.
\begin{description}
\item[Type $A^{(1)}_\ell$]
In this case, the fact that the level of $\psi$ is zero says that
\[
\psi_0+\psi_1+\dots+\psi_\ell=0.
\]
In type $A^{(1)}_\ell$ the $j$th coordinate of $\oa i$ is $2\delta_{ij}-\delta_{i(j-1)}-\delta_{i(j+1)}$ (reading subscripts modulo $\ell+1$).  So if we sum $j$ times the $j$th coordinate over $j$, we get $2i-(i-1)-(i+1)\equiv0$ modulo $\ell+1$.  Hence if $\psi\in\bar Q$ then $\sum_jj\psi_j\equiv 0\pmod{\ell+1}$ as required.

Conversely, suppose $\psi$ is a hub of level $0$ for which $\sum_jj\psi_j$ is divisible by $\ell+1$.  We can eliminate all but two of the coordinates of $\psi$ by subtracting multiples of the $\oa i$.  Specifically (if $\ell>1$) we eliminate the $i$th coordinate of $\psi$ for $i=\ell,\ell-1,\dots,2$ in turn, by subtracting an appropriate multiple of $\oa{i-1}$.

We are left with a hub $\psi'$ which satisfies $\psi'_i=0$ for $i>1$ and which differs from $\psi$ by an element of $\bar Q$.  Since $\psi'$ is of level $0$, we have $\psi'_0+\psi'_1=0$.  And $\psi'$ satisfies the same congruence as $\psi$, so $\psi'_1$ (and hence also $\psi'_0$) is divisible by $\ell+1$.

But now consider the hub $\phi=\oa\ell+2\oa{\ell-1}+3\oa{\ell-2}+\dots+\ell\oa1\in\bar Q$.  For $j\geq2$ we have
\begin{align*}
\phi_j&=-(j-2)+2(j-1)-j=0,\\
\intertext{while}
\phi_1&=-(\ell-1)+2\ell=\ell+1,\\
\phi_0&=-1-\ell.
\end{align*}
Hence we have $\psi'=\frac1{\ell+1}\psi'_1\phi\in\bar Q$ and this, by the construction of $\psi'$ from $\psi$, implies the desired result that $\psi \in \bar Q$.
\item[Type $B^{(1)}_\ell$]
In this case, since the level of $\psi$ is $0$, we have
\[\psi_0+\psi_1+2(\psi_2+\dots+\psi_{\ell-1})+\psi_\ell=0.\]
The coordinates of $\oa i$ are given by the $i$th column of the Cartan matrix
\setlength\lc{0pt}
\setlength\lo{-5pt}
\newcommand\pb{\hbox to 15pt{\vbox to 15pt{\ }}}
\[
\brac{\begin{array}{c@{\hspace{\lc}}c@{\hspace{\lc}}c@{\hspace{\lc}}c@{\hspace{\lc}}c@{\hspace{\lc}}c@{\hspace{\lc}}c}
2&0&\mo&0&\pb&\cdots&0\\[\lo]
0&2&\mo&0&\pb&\cdots&0\\[\lo]
\mo&\mo&2&\mo&0&\pb&\vdots\\[\lo]
0&0&\mo&2&\ddts&\ddts&\pb\\[\lo]
\pb&\pb&0&\ddts&\ddts&\mo&0\\[\lo]
\vdots&\vdots&\pb&\ddts&\mo&2&\mo\\[\lo]
0&0&\cdots&\pb&0&\mt&2
\end{array}}
\]
Since the $\ell$th entry in each column is even, we must have $\psi_\ell \equiv 0 \ppmod 2$ if $\psi \in \bar Q$.

Conversely, suppose $\psi$ has level $0$ and $\psi_\ell$ is even. We eliminate the coordinates of $\psi$ by subtracting integral linear combinations of the $\oa i$, iteratively, as follows.  We eliminate $\psi_\ell$ by subtracting $\frac12\psi_\ell$ times $\oa\ell$, getting a weight $\psi'$ which is still of level $0$ and has $\psi'_\ell=0$. Then for each $i=\ell-1,\ell-2,\dots,2$ we can eliminate the $i$th coordinate by subtracting an appropriate multiple of $\oa{i-1}$. 
This leaves us with a weight $\psi''$ which has $\psi''_i=0$ unless $i=0$ or $1$, and (since $\psi''$ has level $0$) $\psi''_0+\psi''_1=0$.  But now we have $\psi''=\psi''_1(\oa1+\oa2+\dots+\oa\ell)$, as required.

\item[Type $C^{(1)}_\ell$]
In this case, the condition that $\psi$ is of level $0$ becomes
\[\psi_0+\dots+\psi_\ell=0.\]
For this type, the Cartan matrix is as follows.

\[
\brac{\begin{array}{c@{\hspace{\lc}}c@{\hspace{\lc}}c@{\hspace{\lc}}c@{\hspace{\lc}}c@{\hspace{\lc}}c}
2&\mo&0&\pb&\cdots&0\\[\lo]
\mt&2&\mo&0&\pb&\vdots\\[\lo]
0&\mo&2&\ddts&\ddts&\pb\\[\lo]
\pb&0&\mo&\ddts&\mo&0\\[\lo]
\vdots&\pb&\ddts&\ddts&2&\mt\\[\lo]
0&\cdots&\pb&0&\mo&2
\end{array}}
\]
We first note that in each column, the sum of the even-numbered coordinates is even.  Thus if $\psi \in \bar Q$, the required congruence holds.

Conversely, let us assume that the sum of the even-numbered coordinates of $\psi$ is even. For $i=\ell,\ell-1,\dots,2$ in turn we can eliminate the $i$th coordinate of $\psi$ by subtracting an appropriate multiple of $\oa{i-1}$ from $\psi$.  This leaves us with a weight $\psi'$ which has $\psi'_i=0$ for $i>1$; $\psi'$ also satisfies the given congruence, which means that $\psi'_0$ is even, and $\psi'$ has level $0$, which means that $\psi'_0+\psi'_1=0$.  Hence $\psi'=\frac12\psi'_0\oa0\in\bar Q$.

\item[Type $D^{(1)}_\ell$]
Since $\psi$ is of level $0$, we get
\[\psi_0+\psi_1+2(\psi_2+\dots+\psi_{\ell-2})+\psi_{\ell-1}+\psi_\ell = 0.\]
The Cartan matrix in type $D^{(1)}_\ell$ is given by:
\[
\brac{\begin{array}{c@{\hspace{\lc}}c@{\hspace{\lc}}c@{\hspace{\lc}}c@{\hspace{\lc}}c@{\hspace{\lc}}c@{\hspace{\lc}}c@{\hspace{\lc}}c}
2&0&\mo&0&\pb&\cdots&0&0\\[\lo]
0&2&\mo&0&\pb&\cdots&0&0\\[\lo]
\mo&\mo&2&\mo&0&\pb&\vdots&\vdots\\[\lo]
0&0&\mo&2&\ddts&\ddts&\pb&\pb\\[\lo]
\pb&\pb&0&\ddts&\ddts&\mo&0&0\\[\lo]
\vdots&\vdots&\pb&\ddts&\mo&2&\mo&\mo\\[\lo]
0&0&\cdots&\pb&0&\mo&2&0\\[\lo]
0&0&\cdots&\pb&0&\mo&0&2
\end{array}}.
\]
We observe that in each column the sum of the first two entries is even, so if $\psi\in\bar Q$ then certainly $\psi_0+\psi_1$ is even.  Furthermore, if $\ell$ is even, then the sum of the odd-numbered entries in each column is even, so $\psi$ satisfies the required congruences in this case.  A similar check guarantees the required congruence in the case where $\ell$ is odd.

Conversely, suppose $\psi$ does satisfy the required congruences.  Then in particular $\psi_0$ and $\psi_1$ have the same parity. By subtracting $\oa2$ from $\psi$ if necessary, we may assume that $\psi_0,\psi_1$ are even.  Now we can eliminate these two coordinates by subtracting appropriate multiples of $\oa0,\oa1$ from $\psi$.  Having done this, we eliminate the $i$th coordinate of $\psi$ for $i=2,3,\dots,\ell-2$ in turn by subtracting an appropriate multiple of $\oa{i+1}$.  This leaves us with a hub $\psi'$ having $\psi'_i=0$ for $i<\ell-1$.  $\psi'$ is of level $0$, so $\psi'_{\ell-1}+\psi'_\ell=0$; furthermore, $\psi'$ satisfies the same congruence as $\psi$, which implies (whether $\ell$ is even or odd) that $\psi'_{\ell-1}$ is even.  But now we have $\psi'=\frac12\psi'_{\ell-1}(\oa{\ell-1}-\oa\ell)\in\bar Q$.

\item[Type $D^{(2)}_{\ell+1}$]
Suppose $\fkg$ is of type $D^{(2)}_{\ell+1}$.   Since $\psi$ is of level $0$, it satisfies
\[\psi_0+2(\psi_1+\dots+\psi_{\ell-1})+\psi_\ell=0.\]
We are trying to prove that $\psi \in \bar Q$ if and only if $\psi_0\equiv 0\pmod2$.  The Cartan matrix is
\[
\brac{\begin{array}{c@{\hspace{\lc}}c@{\hspace{\lc}}c@{\hspace{\lc}}c@{\hspace{\lc}}c@{\hspace{\lc}}c@{\hspace{\lc}}c}
2&\mt&0&\pb&\cdots&0\\[\lo]
\mo&2&\mo&0&\pb&\vdots\\[\lo]
0&\mo&2&\mo&\ddts&\pb\\[\lo]
\pb&0&\ddts&\ddts&\ddts&0\\[\lo]
\vdots&\pb&\ddts&\mo&2&\mo\\[\lo]
0&\cdots&\pb&0&\mt&2
\end{array}}.
\]

Since the top row of this matrix consists of even numbers, we must have $\psi_0\equiv0\ppmod2$ if $\psi \in \bar Q$.

Conversely, suppose $\psi_0$ is even.  By subtracting an appropriate multiple of $\oa0$ from $\psi$ if necessary, we can assume that $\psi_0=0$.  Then we can eliminate the $i$th coordinate of $\psi$ for $i=1,\dots,\ell-1$ in turn by subtracting an appropriate multiple of $\oa{i+1}$.  We are left with a hub all of whose coordinates except possibly the last are zero.  Since this hub has level $0$, the last coordinate must be zero too.

\item[Type $A^{(2)}_{2\ell-1}$]
 Suppose $\fkg$ is of type $A^{(2)}_{2\ell-1}$.  Since $\psi$ is of level $0$, we get
\[\psi_0+\psi_1+2(\psi_2+\dots+\psi_\ell)=0.\]
Each $\oa i$ is determined by the appropriate column of the  Cartan matrix
\[
\brac{\begin{array}{c@{\hspace{\lc}}c@{\hspace{\lc}}c@{\hspace{\lc}}c@{\hspace{\lc}}c@{\hspace{\lc}}c@{\hspace{\lc}}c}
2&0&\mo&0&\pb&\cdots&0\\[\lo]
0&2&\mo&0&\pb&\cdots&0\\[\lo]
\mo&\mo&2&\mo&0&\pb&\vdots\\[\lo]
0&0&\mo&2&\ddts&\ddts&\pb\\[\lo]
\pb&\pb&0&\ddts&\ddts&\mo&0\\[\lo]
\vdots&\vdots&\pb&\ddts&\mo&2&\mt\\[\lo]
0&0&\cdots&\pb&0&\mo&2
\end{array}}.
\]
In any column the sum of the odd-numbered entries is even (regardless of whether $\ell$ is even), so the required condition holds if $\psi \in \bar Q$.

Conversely, suppose that the condition holds.  Since $\psi$ has level $0$, we must have $\psi_0+\psi_1$ even; by subtracting $\oa2$ if necessary, we may assume $\psi_0$ and $\psi_1$ are both even.  Then we can eliminate these coordinates by subtracting appropriate multiples of $\oa0,\oa1$.  Then we eliminate the $i$th coordinate for $i=2,\dots,\ell-2$ by subtracting an appropriate multiple of $\oa{i-1}$.  We are left with a hub $\psi'$ whose $i$th coordinate is zero for $i<\ell-1$.  Since $\psi'$ has level $0$, we must have $\psi'_{\ell-1}+\psi'_\ell=0$.  Moreover, the sum of the odd-numbered coordinates of $\psi'$ is even, so $\psi'_{\ell-1}$ and $\psi'_\ell$ are even.  But now $\psi'=\frac12\psi'_\ell\oa\ell\in\bar Q$.

\item[Type $E^{(1)}_{6}$]   Since $\psi$ is of level $0$, it satisfies
\[\psi_0+2\psi_6+\psi_1+2\psi_2+3\psi_3+2\psi_4+\psi_5=0.\]
The Cartan matrix is as follows.

\setlength\lc{3.5pt}\setlength\lo{-0.5pt}
\[
\brac{\begin{array}{c@{\hspace{\lc}}c@{\hspace{\lc}}c@{\hspace{\lc}}c@{\hspace{\lc}}c@{\hspace{\lc}}c@{\hspace{\lc}}c}
2&0&0&0&0&0&\mo\\[\lo]
0&2&\mo&0&0&0&0\\[\lo]
0&\mo&2&\mo&0&0&0\\[\lo]
0&0&\mo&2&\mo&0&\mo\\[\lo]
0&0&0&\mo&2&\mo&0\\[\lo]
0&0&0&0&\mo&2&0\\[\lo]
\mo&0&0&\mo&0&0&2
\end{array}}.
\]
We are trying to prove that $\psi \in \bar Q$ if and only if
\[\psi_0+2\psi_6\equiv \psi_5+2\psi_4\pmod 3.\]
We first note that this condition does indeed hold for all the $\oa{i}$, and thus for every element of $\bar Q$.  In order to prove the opposite direction of the implication, we first replace $\psi$ with a $\bar Q$-equivalent hub $\psi'$ in which the $0$th, $1$st and $5$th coordinates have been eliminated:
\[\psi'=\psi+\psi_0\oa6+\psi_1 \oa2+\psi_5 \oa4.\]
The three elements of $\bar Q$ given by $\oa0+2\oa6$,$\oa1+2\oa2$, and $\oa5+2\oa4$ each have a $\mt$ in the third row, a $3$ in one of the rows $2$, $4$ or $6$, and $0$s in all other rows.  By adding suitable multiples of these elements to $\psi'$, we can create a new hub $\psi''$ in which $\psi''_0=\psi''_1=\psi''_5=0$ and $\psi''_2,\psi''_4,\psi''_6\in\{0,1,2\}$.  $\psi''$ satisfies the same congruence as $\psi$, which then means that $\psi''_4=\psi''_6$.

$\psi''$ also has level $0$, which tells us that
\[
2(\psi_2''+\psi_4''+\psi_6'')=-3\psi_3'';
\]
in particular, $\psi''_2+\psi''_4+\psi''_6$ is divisible by $3$.  Since $\psi''_4=\psi''_6$, this means that $\psi''_2$, $\psi''_4$, $\psi_6''$ are all equal, and thus $\psi''$ is a multiple of $\oa3$ and so is in $\bar Q$.

\item[Type $E^{(1)}_{7}$]   Since $\psi$ is of level $0$, it satisfies
\[\psi_0+2\psi_1+3\psi_2+4\psi_3+3\psi_4+2\psi_5+\psi_6+2\psi_7=0,\]
The Cartan matrix is
\[
\brac{\begin{array}{c@{\hspace{\lc}}c@{\hspace{\lc}}c@{\hspace{\lc}}c@{\hspace{\lc}}c@{\hspace{\lc}}c@{\hspace{\lc}}c@{\hspace{\lc}}c}
2&\mo&0&0&0&0&0&0\\[\lo]
\mo&2&\mo&0&0&0&0&0\\[\lo]
0&\mo&2&\mo&0&0&0&0\\[\lo]
0&0&\mo&2&\mo&0&0&\mo\\[\lo]
0&0&0&\mo&2&\mo&0&0\\[\lo]
0&0&0&0&\mo&2&\mo&0\\[\lo]
0&0&0&0&0&\mo&2&0\\[\lo]
0&0&0&\mo&0&0&0&2
\end{array}}.
\]
We are trying to prove that $\psi \in \bar Q$ if and only if $\psi_0+\psi_2+\psi_7$ is even.  Summing rows $0$, $2$ and $7$ shows that every element of $\bar Q$ does indeed satisfy this congruence.

We now assume that $\psi$ satisfies the congruence and show that it lies in $\bar Q$. We being by eliminating the first two coordinates of $\psi$ by adding appropriate multiples of $\oa1$ and $\oa2$.  The resulting hub $\psi'$ also satisfies the given congruence, so that $\psi'_2$ and $\psi'_7$ have the same parity.  By adding an appropriate multiple of $\oa7$, we can assume that $\psi'_2=\psi'_7$.  Now we can eliminate both of these coordinates by adding a multiple of $\oa3$.  Then we can eliminate the $3$rd, $4$th and $5$th coordinates in succession by adding appropriate multiples of $\oa4$, $\oa5$ and $\oa6$.  We are left with a hub $\psi''$ in which all coordinates except possibly $\psi''_6$ are zero.  Since $\psi''$ has level zero, $\psi''_6$ is zero too, and so $\psi\in\bar Q$.

\item[Type $E^{(1)}_{8}$]   Since $\psi$ is of level $0$, it satisfies
\[\psi_0+2\psi_1+3\psi_2+4\psi_3+5\psi_4+6\psi_5+4\psi_6+2\psi_7+3\psi_8=0.\]
In this case there is no additional congruence. The Cartan matrix is
\[
\brac{\begin{array}{c@{\hspace{\lc}}c@{\hspace{\lc}}c@{\hspace{\lc}}c@{\hspace{\lc}}c@{\hspace{\lc}}c@{\hspace{\lc}}c@{\hspace{\lc}}c@{\hspace{\lc}}c}
2&\mo&0&0&0&0&0&0&0\\[\lo]
\mo&2&\mo&0&0&0&0&0&0\\[\lo]
0&\mo&2&\mo&0&0&0&0&0\\[\lo]
0&0&\mo&2&\mo&0&0&0&0\\[\lo]
0&0&0&\mo&2&\mo&0&0&0\\[\lo]
0&0&0&0&\mo&2&\mo&0&\mo\\[\lo]
0&0&0&0&0&\mo&2&\mo&0\\[\lo]
0&0&0&0&0&0&\mo&2&0\\[\lo]
0&0&0&0&0&\mo&0&0&2
\end{array}}.
\]
In this case we will eliminate all the coordinates starting at the end.  We will get a series of intermediate hubs, which will be denoted by $\psi=\psi^0, \psi^1,\dots,\psi^7$, where for each $\psi^i$, only $\psi^i_0,\dots,\psi^i_{8-i}$ can be non-zero.  The modifications by elements of $\bar Q$ are as follows:
\begin{align*}
\psi^1&=\psi^0+\psi^0_7 \oa{6},\\
\psi^2&=\psi^1+\psi^1_6\oa{5},\\
\psi^3&=\psi^2-\psi^2_8(3\oa{5}+2\oa{6}+\oa{7}+2\oa{8}),\\
\psi^4&=\psi^3+\psi^3_5 \oa{4},\\
\psi^5&=\psi^4+\psi^4_4 \oa{3},\\
\psi^6&=\psi^5+\psi^5_3 \oa{2},\\
\psi^7&=\psi^6+\psi^6_2 \oa{1}.
\end{align*}
The level 0 condition then implies that $\psi_8$ is a multiple of $\oa0 \in \bar Q$, from which we conclude that $\psi \in \bar Q$, as desired.

\item[Types $A^{(2)}_{2\ell}$, $E^{(2)}_6$ and $D^{(3)}_4$]
In these cases, there is no additional condition in the table, so we must show that the level $0$ condition alone guarantees that $\psi\in\bar Q$. In each of these cases, we examine the Cartan matrix and find that for each $i=1,\dots,\ell$ we have $(\oa{i-1})_i=-1$ while $(\oa{i-1})_j=0$ for $j>i$; in terms of the Dynkin diagram, this just says that the diagram is a path, with all the arrows pointing to the left.  Hence given $\psi$ of level $0$, we may eliminate the $i$th coordinate of $\psi$ for $i=\ell,\ell-1,\dots,1$ in turn by subtracting an appropriate multiple of $\oa{i-1}$.  We are left with a hub $\psi'$ which satisfies $\psi'_i=0$ for $i>0$.  Since $\psi'$ is also of level $0$, we get $\psi'=0$ too, because the level $0$ condition always involves $\psi_0$ with non-zero coefficient.
\item[Types $F^{(1)}_4$ and $G^{(1)}_2$]
These cases are dealt with in the same way as the preceding cases, except that we reverse the ordering of the nodes: we now have $(\oa i)_{i-1}=-1$ for $i=1,\dots,\ell$, while $(\oa i)_j=0$ for $j<i-1$.  So given a hub $\psi$ of level $0$, we may eliminate the $i$th coordinate of $\psi$ for $i=0,\dots,\ell-1$ in turn by subtracting an appropriate multiple of $\oa{i+1}$.  We are left with a hub $\psi'$ of level $0$ which satisfies $\psi'_i=0$ for $i<\ell$, and which must therefore be zero.
\qedhere\end{description}
\end{proof}

\section{The defect of a weight in $P(\Lm)$}

If $\eta=\Lm-\alpha$ is a weight and $s$ is an integer, we recall that
\[
\defe(\eta)=(\eta|\alpha)-\tfrac12(\alpha|\alpha).
\]
then
\[
\defe(\eta+s\delta)=\defe(\eta)+s(\Lambda|\delta)=\defe(\eta)+sk,
\]
since
\[
(\Lm-\alpha-\delta|\Lm-\alpha-\delta)=(\Lm-\alpha|\Lm-\alpha)-2(\Lm-\alpha|\delta)+(\delta|\delta)
\]
and
\[(\delta|\delta)=(\alpha|\delta)=(\delta|\alpha)=0, \quad(\Lambda|\delta)=k.\]

From these results we see that the defect of a weight $\eta$ of level $k$ is at least $k$ times its $\delta$-shift $s$.
Since by \cite[Lemma 12.6]{Ka}, every maximal weight in level $1$ is in $W\cdot\Lm$ and thus of defect $0$, this shows that for level $1$, the defect is equal to the $\delta$-shift $s$.

There is a unique dominant weight of defect zero, namely $\Lm$.  Indeed, the defect of $\eta$ is zero only if $(\Lm|\Lm)=(\eta|\eta)$, and by \cite[Prop.\ 11.4]{Ka} this implies that $\eta$ lies in the $W$-orbit of $\Lm$.   Since any $W$-orbit contains a unique dominant weight, we see that if $\eta$ has defect zero and a positive hub, then $\eta =\Lm$.

\begin{remark} In the level one case, as just described, one can determine the maximal weights because they all have defect zero.
For higher levels, the defect is not sufficient to determine whether or not a weight is a maximal weight.  Although any weight of defect less that $k$ must be a maximal weight, there may be weights of defect greater than $k$ which are maximal weights.  If, for example, $\mathfrak g = \fkg(A^{(1)}_\ell)$, $\ell\geq2$ and $\Lm = 2m\Lm_0$, then there is a weight $\eta=\Lm - m\alpha_0$ with positive hub $(0,m,0,\dots,0,m)$ and defect
$$\defe(\eta)=(2m\Lm_0|m\alpha_0)-\tfrac{1}{2}(m\alpha_0|m\alpha_0)=2m^2-m^2=m^2.$$ For $m>2$, this is larger than $k=2m$.
\end{remark}

\section{Finding the weights in $P(\Lm)$ -- an example}\label{s:examp}

\begin{figure}[b]
\[
\begin{tikzpicture}[scale=1.3]
\draw(0,0)node{\begin{tabular}{c}$(0,3,3)$\\\small$[7,-1,-2]$\end{tabular}};
\draw(2,1)node{\begin{tabular}{c}$(0,2,3)$\\\small$[6,1,-3]$\end{tabular}};
\draw(-4,4)node{\begin{tabular}{c}$(0,2,0)$\\\small$[3,-2,3]$\end{tabular}};
\draw(0,6)node{\begin{tabular}{c}$(0,0,0)$\\\small$[1,2,1]$\end{tabular}};
\draw(2,5)node{\begin{tabular}{c}$(0,0,1)$\\\small$[2,3,-1]$\end{tabular}};
\draw(-4,2)node{\begin{tabular}{c}$(0,3,1)$\\\small$[5,-3,2]$\end{tabular}};
\foreach \y in {0.3}
{\draw[<->](2*\y,\y)--(2-2*\y,1-\y);
\draw[<->](-4+2*\y,4-\y)--(2-2*\y,1+\y);
\draw[<->](-4+2*\y,4+\y)--(-2*\y,6-\y);
\draw[<->](2*\y,6-\y)--(2-2*\y,5+\y);
\draw[<->](2-2*\y,5-\y)--(-4+2*\y,2+\y);
\draw[<->](-4+2*\y,2-\y)--(-2*\y,\y);
}
\draw(1,.5)node[fill=white,inner sep=1pt]{\small$1$};
\draw(-1,3.5)node[fill=white,inner sep=1pt]{\small$1$};
\draw(-2,5)node[fill=white,inner sep=1pt]{\small$1$};
\draw(-2,1)node[fill=white,inner sep=1pt]{\small$2$};
\draw(-1,2.5)node[fill=white,inner sep=1pt]{\small$2$};
\draw(1,5.5)node[fill=white,inner sep=1pt]{\small$2$};
\end{tikzpicture}
\]
\begin{center}\caption{}\label{fig2}\end{center}
\end{figure}
In this section we give an example in which we calculate the maximal dominant weights in $P(\Lm)$ use them to test whether a weight $\eta$ lies in $P(\Lm)$.

Recall that the Weyl group $W$ may be regarded as a group of isometries of the weight space of $\fkg$; the generating reflections $s_0,\dots,s_\ell$ act via
\[
s_i:\eta\longmapsto \eta-\la h_i,\eta\ra\alpha_i.
\]

The results in this paper give a way to determine whether a given integral weight $\eta$ lies in the weight space $P(\Lm)$.  To do this, we begin by finding the set $N$ of all maximal dominant weights in $P(\Lm)$, and then computing $\tilde N=\ow \cdot N$.

In this section we study the example where $\mathfrak g =\widehat{\mathfrak{sl}}_3$ (that is, of type $A^{(1)}_2$) and $\Lm=\Lm_0+2\Lm_1+\Lm_2$ (so $\ell=2$ and $k=4$).  We represent a weight $\Lm-\gamma_0\alpha_0-\gamma_1\alpha_1-\gamma_2\alpha_2$ by its content $(\gamma_0,\gamma_1,\gamma_2)$, and we write the hub as $[\theta_0,\theta_1,\theta_2]$.

To begin with, we find the positive hubs, using Proposition \ref{smallpath}.  The hub of $\Lm$ is $[1,2,1]$, of level $4$.  So the hubs of weights in $P(\Lm)$ are those $\theta$ for which $\theta_0+\theta_1+\theta_2=4$ and $\theta_1+2\theta_2\equiv1\ppmod3$.  It is easy to check that the only positive hubs satisfying these conditions are
\[
[1,2,1],\quad[2,0,2],\quad[3,1,0],\quad[0,4,0],\quad[0,1,3].
\]
Using Proposition \ref{maxcont}, we can find the corresponding maximal weights, which we label $a^0,a^1,a^2,A^1,A^2$, respectively.  The contents of these weights are as follows:
\begin{alignat*}2
\gamma(a^0)&=(0,0,0),&\quad\gamma(a^1)&=(0,1,0),\quad\gamma(a^2)=(0,1,1),\\
\qquad\gamma(A^1)&=(1,0,1),&\quad\gamma(A^2)&=(1,1,0).
\end{alignat*}
The superscript in the notation for each weight indicates the defect.

Next we compute $\tilde N$ by applying the reflections $r_{\alpha_1},r_{\alpha_2}$ to these five weights.  Applying $r_{\alpha_i}$ to a weight means adding $\theta_i$ copies of $-\alpha_i$, where $\theta_i$ is the $i$th component of the hub.  For example, for the weight $a^0$, we get the picture in Figure \ref{fig2} (where we write both the content and hub of each weight, and arrows labelled $i$ represent the reflections $s_i=r_{\alpha_i}$).

It turns out that $\tilde N$ contains 21 weights, comprising four $\ow$-orbits.  We describe these, together with their images under the reflections $r_{\alpha_1},r_{\alpha_2}$, in the following table.
\begin{longtable}{|c|c|c|ccc|l|}\hline
$\eta$&$\text{content}$&$\text{hub}$&$r_{\alpha_1}(\eta)$&$r_{\alpha_2}(\eta)$&$\ \ \ \tilde \eta$\\\hline\endhead
$a^0$&$(0,0,0)$&$[1,2,1]$&$b^0$&$d^0$&$(0,0)$\\
$b^0$&$(0,2,0)$&$[3,-2,3]$&$a^0$&$c^0$&$(2,0)$\\
$c^0$&$(0,2,3)$&$[6,1,-3]$&$f^0$&$b^0$&$(0,3)$\\
$d^0$&$(0,0,1)$&$[2,3,-1]$&$e^0$&$a^0$&$(0,1)$\\
$e^0$&$(0,3,1)$&$[5,-3,2]$&$d^0$&$f^0$&$(3,1)$\\
$f^0$&$(0,3,3)$&$[7,{-}1,-2]$&$c^0$&$e^0$&$(3,3)$\\\hline
$a^1$&$(0,1,0)$&$[2,0,2]$&$a^1$&$b^1$&$(1,0)$\\
$b^1$&$(0,1,2)$&$[4,2,-2]$&$c^1$&$a^1$&$(1,2)$\\
$c^1$&$(0,3,2)$&$[6,-2,0]$&$b^1$&$c^1$&$(3,2)$\\\hline
$a^2$&$(0,1,1)$&$[3,1,0]$&$b^2$&$a^2$&$(1,1)$\\
$b^2$&$(0,2,1)$&$[4,-1,1]$&$a^2$&$c^2$&$(2,1)$\\
$c^2$&$(0,2,2)$&$[5,0,-1]$&$c^2$&$b^2$&$(2,2)$\\\hline
$A^1$&$(1,0,1)$&$[0,4,0]$&$B^1$&$A^1$&$(3,0)^\ast$\\
$B^1$&$(1,4,1)$&$[4,-4,4]$&$A^1$&$C^1$&$(3,0)$\\
$C^1$&$(1,4,5)$&$[8,0,-4]$&$C^1$&$B^1$&$(3,0)$\\\hline
$A^2$&$(1,1,0)$&$[0,1,3]$&$B^2$&$D^2$&$(0,3)^\ast$\\
$B^2$&$(1,2,0)$&$[1,-1,4]$&$A^2$&$C^2$&$(1,3)^\ast$\\
$C^2$&$(1,2,4)$&$[5,3,-4]$&$F^2$&$B^2$&$(1,3)$\\
$D^2$&$(1,1,3)$&$[3,4,-3]$&$E^2$&$A^2$&$(0,2)^\ast$\\
$E^2$&$(1,5,3)$&$[7,-4,1]$&$D^2$&$F^2$&$(0,2)$\\
$F^2$&$(1,5,4)$&$[8,-3,-1]$&$C^2$&$E^2$&$(0,3)$\\\hline
\end{longtable}
In the cases where the $T$-class has more than one element, we have chosen the first occurrence in the table as representative and marked it with an asterisk.  Using this table, one can then test a given weight $\eta\in\Lm-Q$ to see whether it lies in $P(\Lm)$, by calculating $\tilde \eta$.  For example, if $\eta = (2,7,3)$, we compute
\[
\tilde\eta=((7-2)\,\amod 4,(3-2)\,\amod 4)=(1,3)=\tilde a^2.
\]
Letting $\zeta=a^2=\Lm-\alpha_1-\alpha_2$, we have
$$\eta-\zeta=-2\alpha_0-6\alpha_1-2\alpha_2.$$
Hence $\eta=t_\alpha(\zeta)$, where $\alpha={-}\alpha_1$.  To determine whether or not $\xi$ is in $P(\Lm)$, we now calculate (from Theorem \ref{main})
\[
s(\eta)=2-\left((\zeta|{-}\alpha_1)-\tfrac12({-}\alpha_1|{-}\alpha_1)4\right).
\]
Since $(\zeta|{-}\alpha_1)=-2+2-1=-1$, and $({-}\alpha_1|{-}\alpha_1)=2$, we get $s(\eta)=-1<0$, showing that $\eta$ is not in $P(\Lm)$ by Lemma \ref{moddelta}.

\begin{figure}[phtb]\label{crystal}
\begin{center}
\setlength\ofs{0.5mm}
\setlength\rd{1.0cm}
\setlength\ld{1.0cm}
\setlength\dd{1.0cm}
\setlength\ri{2.0cm}
\setlength\lf{1.1cm}

\[
\begin{tikzpicture}[scale=1.146]
\draw[bluex](-\lf,-\ld-2\dd)--++(-2\lf,-2\ld)--++(6\ri,-6\rd)--++(-4\lf,-4\ld)--++(-2\ri,2\rd)--++(6\lf,6\ld)--cycle;
\draw[bluex](\ri,-\rd-2\dd)--++(\ri,-\rd)--++(-6\lf,-6\ld)--++(5\ri,-5\rd)--++(\lf,\ld)--++(-6\ri,6\rd)--cycle;
\draw[bluex](-2\lf,-2\ld-2\dd)--++(5\ri,-5\rd)--++(-5\lf,-5\ld);
\draw[bluex](-4\lf+\ri,-2\dd-4\ld-\rd)--++(5\ri,-5\rd);
\draw[bluex](3\ri-\lf,-2\dd-3\rd-\ld)--++(-5\lf,-5\ld);
\draw[greenx](0,-\dd)--++(-3\lf,-3\ld)--++(5\ri,-5\rd)--++(-2\lf,-2\ld)--++(-3\ri,3\rd)--++(5\lf,5\ld)--cycle;
\draw[greenx](-\lf,-\dd-\ld)--++(3\ri,-3\rd)--++(-4\lf,-4\ld);
\draw[greenx](-2\lf,-\dd-2\ld)--++(4\ri,-4\rd)--++(-3\lf,-3\ld);
\draw[greenx](\ri,-\rd-\dd)--++(-4\lf,-4\ld)--++(4\ri,-4\rd);
\foreach \x in {0,1,2}
\draw[greyx](-3\lf-\x\lf+\x\ri,-3\ld-\x\ld-\x\rd-\dd)--++(0,-\dd);
\foreach \x in {3,4,5}
\draw[greyx](-5\lf+\x\ri,-5\ld-\x\rd-\dd)--++(0,-\dd);
\draw[greyx](-4\lf+5\ri,-4\ld-5\rd-\dd)--++(0,-\dd);
\foreach \x in {0,1,2,3}
\draw[greyx](-\x\lf+2\ri+\x\ri,-\x\ld-2\rd-\x\rd-\dd)--++(0,-\dd);
\draw[redx](0,0)--++(\ri,-\rd)--++(-3\lf,-3\ld)--++(2\ri,-2\rd)--++(\lf,\ld)--++(-3\ri,3\rd)--cycle;
\draw[redx](-\lf,-\ld)--++(2\ri,-2\rd)--++(-2\lf,-2\ld);
\draw[greyx](-2\lf,-2\ld)--++(0,-\dd);
\foreach \x in {1,2,3}
\draw[greyx](-3\lf+\x\ri,-3\ld-\x\rd)--++(0,-\dd);
\foreach \x in {0,1,2}
\draw[greyx](\ri+\x\ri-\x\lf,-\rd-\x\rd-\x\ld)--++(0,-\dd);
\draw(-\lf,-\ld-2\dd)node[bluex,fill=white,inner sep=0pt]{\nod2100};
\draw(-2\lf,-2\ld-2\dd)node[bluex,fill=white,inner sep=0pt]{\nod2200};
\draw(-3\lf,-3\ld-2\dd)node[bluex,fill=white,inner sep=0pt]{\nod2300};
\draw(-5\lf+\ri,-5\ld-\rd-2\dd)node[bluex,fill=white,inner sep=0pt]{\nod2510};
\draw(-6\lf+2\ri,-6\ld-2\rd-2\dd)node[bluex,fill=white,inner sep=0pt]{\nod2620};
\draw(-7\lf+4\ri,-7\ld-4\rd-2\dd)node[bluex,fill=white,inner sep=0pt]{\nod2740};
\draw(-7\lf+5\ri,-7\ld-5\rd-2\dd)node[bluex,fill=white,inner sep=0pt]{\nod2750};
\draw(-7\lf+6\ri,-7\ld-6\rd-2\dd)node[bluex,fill=white,inner sep=0pt]{\nod2760};
\draw(-6\lf+7\ri,-6\ld-7\rd-2\dd)node[bluex,fill=white,inner sep=0pt]{\nod2670};
\draw(-5\lf+7\ri,-5\ld-7\rd-2\dd)node[bluex,fill=white,inner sep=0pt]{\nod2570};
\draw(-3\lf+6\ri,-3\ld-6\rd-2\dd)node[bluex,fill=white,inner sep=0pt]{\nod2360};
\draw(\ofs-2\lf+5\ri,-2\ld-5\rd-2\dd)node[bluex,fill=white,inner sep=0pt]{\nod2250};
\draw(\ofs-\lf+4\ri,-\ld-4\rd-2\dd)node[bluex,fill=white,inner sep=0pt]{\nod2140};
\draw(2\ri,-2\rd-2\dd)node[bluex,fill=white,inner sep=0pt]{\nod2020};
\draw(-2\ofs+\ri,-\rd-2\dd)node[bluex,fill=white,inner sep=0pt]{\nod2010};
\draw(\ofs,-\dd)node[greenx,fill=white,inner sep=0pt]{\nod1000};
\draw(\ofs-\lf,-\ld-\dd)node[greenx,fill=white,inner sep=0pt]{\nodhub1102013};
\draw(-2\ofs-2\lf,-2\ld-\dd)node[greenx,fill=white,inner sep=0pt]{\nod1202};
\draw(-2\ofs-3\lf,-3\ld-\dd)node[greenx,fill=white,inner sep=0pt]{\nod1300};
\draw(-4\lf+\ri,-4\ld-\rd-\dd)node[greenx,fill=white,inner sep=0pt]{\nod1411};
\draw(-5\lf+2\ri,-5\ld-2\rd-\dd)node[greenx,fill=white,inner sep=0pt]{\nod1520};
\draw(-5\lf+3\ri,-5\ld-3\rd-\dd)node[greenx,fill=white,inner sep=0pt]{\nod1532};
\draw(-5\lf+4\ri,-5\ld-4\rd-\dd)node[greenx,fill=white,inner sep=0pt]{\nod1542};
\draw(-5\lf+5\ri,-5\ld-5\rd-\dd)node[greenx,fill=white,inner sep=0pt]{\nod1550};
\draw(-4\lf+5\ri,-4\ld-5\rd-\dd)node[greenx,fill=white,inner sep=0pt]{\nod1451};
\draw(-\ofs-3\lf+5\ri,-3\ld-5\rd-\dd)node[greenx,fill=white,inner sep=0pt]{\nod1350};
\draw(-\ofs-2\lf+4\ri,-2\ld-4\rd-\dd)node[greenx,fill=white,inner sep=0pt]{\nod1243};
\draw(\ofs-\lf+3\ri,-\ld-3\rd-\dd)node[greenx,fill=white,inner sep=0pt]{\nod1130};
\draw(2\ofs+2\ri,-2\rd-\dd)node[greenx,fill=white,inner sep=0pt]{\nod1020};
\draw(\ofs+\ri,-\rd-\dd)node[greenx,fill=white,inner sep=0pt]{\nodhub1011040};
\draw(0,0)node[redx,fill=white,inner sep=0pt]{\nodhub0000121};
\draw(-\ofs-\lf,-\ld)node[redx,fill=white,inner sep=0pt]{\nodhub0101202};
\draw(-\ofs-2\lf,-2\ld)node[redx,fill=white,inner sep=0pt]{\nod0200};
\draw(\ri,-\rd)node[redx,fill=white,inner sep=0pt]{\nod0010};
\draw(-\ofs+\ri-\lf,-\rd-\ld)node[redx,fill=white,inner sep=0pt]{\nodhub0112310};
\draw(2\ofs+\ri-2\lf,-\rd-2\ld)node[redx,fill=white,inner sep=0pt]{\nod0213};
\draw(2\ofs+\ri-3\lf,-\rd-3\ld)node[redx,fill=white,inner sep=0pt]{\nod0310};
\draw(2\ri-\lf,-2\rd-\ld)node[redx,fill=white,inner sep=0pt]{\nod0120};
\draw(2\ri-2\lf,-2\rd-2\ld)node[redx,fill=white,inner sep=0pt]{\nod0222};
\draw(2\ri-3\lf,-2\rd-3\ld)node[redx,fill=white,inner sep=0pt]{\nod0321};
\draw(-\ofs+3\ri-2\lf,-3\rd-2\ld)node[redx,fill=white,inner sep=0pt]{\nod0230};
\draw(3\ri-3\lf,-3\rd-3\ld)node[redx,fill=white,inner sep=0pt]{\nod0330};
\end{tikzpicture}
\]

\caption{}
\end{center}
\end{figure}

In Figure \ref{crystal} we have given a three-dimensional representation of $P(\Lm)$; the diagonal lines indicate subtraction of $\alpha_1$ and $\alpha_2$, while the vertical lines indicate subtraction of $\alpha_0$; for clarity, we have only drawn the vertical lines  which would be visible in an opaque three-dimensional model.  We have recorded the contents of the maximal weights, with an exponent giving the defect of each maximal weight.  The contents and defects of the other weights can be deduced by shifting by $\delta$.   Since all the positive  hubs have defect less than $k$, one can get the $\delta$-shifts by subtracting the residue of the defect modulo $e$ and dividing by $k$.   For the five maximal weights with positive  hubs we have also indicated the hubs in square brackets.  The reflection $r_{\alpha_1}$ inverts strings going diagonally down to the left, while the reflection $r_{\alpha_2}$ inverts strings going diagonally down to the right.  All the weights along a given horizontal line in the two-dimensional representation have the same height, so in the corresponding cyclotomic Hecke algebra they correspond to the blocks of $H^\Lm_d$ for a fixed rank $d$.

We define the $j$th \textit{floor} of $P(\Lm)$ to be the set of all weights with $0$-content $j$, then each floor is a union of $\ow$-orbits.  The reflection $r_{\alpha_i}$ reflects all $i$-strings for $i=1,\dots,\ell$.

Each floor contains a $\delta$-shifted copy of the previous floor and whatever new maximal weights appear on that floor. The weights on the $0$th floor are all maximal weights, and the remaining maximal weights appear on $0$-strings or at the ends of strings on other floors. We shall see in Lemma \ref{internal} below that there cannot be more than $\max{a_i}$ maximal weights at either end of a string, unless all the weights in the string are maximal.

\section{$\delta$-shifts}

In this section we examine further the $\delta$-shift of a weight in $P(\Lm)$; in particular, we examine how $\delta$-shift varies along a string in $P(\Lm)$.

In Lemma \ref{moddelta} we showed that if $\Lm$ is a dominant integral weight, and $\eta$ is any other equivalent weight, i.e.\ $\eta \in \Lm - Q$, then there exists an integer $s$ such that $\eta +s\delta \in P(\Lm)$.
 The $\delta$-shift  $s(\eta)$ of $\eta\in P(\Lm)$ was defined as the largest $s$ such that $\eta+s\delta\in P(\Lm)$.  So $s(\eta)=0$ if and only if $\eta$ is maximal.

Recall that $\Delta^{\re}$ denotes the set of \emph{real roots}, i.e.\ the images of the simple roots under the action of the Weyl group.  For $\alpha\in\Delta^{\re}$, an \emph{$\alpha$-string} is a set of weights
\[\lm,\lm+\alpha,\lm+2\alpha,\dots,\lm+t\alpha\]
all lying in $P(\Lm)$, with $\lm-\alpha,\lm+(t+1)\alpha\notin P(\Lm)$.  If $\alpha$ is the simple root $\alpha_i$, then we call an $\alpha$-string an \emph{$i$-string}.

\begin{prop}\label{internal}
Suppose $\alpha\in\Delta^{\re}$ and $\eta\in P(\Lm)$. Let $a=\max\{a_i\mid i=0,\dots,\ell\}$. Suppose that $s(\eta)=0$, and that $\eta+a\alpha,\eta-a\alpha$ both lie in $P(\Lm)$. Then every weight in the $\alpha$-string containing $\eta$ has $\delta$-shift $0$.
\end{prop}

\begin{proof}
By \cite[Corollary 10.1]{Ka}, we can find $w\in W$ such that $w\eta$ has a positive hub.  Since the action of $W$ preserves $P(\Lm)$, the $\alpha$-string $S$ containing $\eta$ is mapped by $w$ to the $(w\alpha)$-string $wS$ containing $w\eta$ as well as $w\eta\pm aw\alpha$.  Furthermore, since $\delta$ is fixed by the action of $W$, the $k$-values along $wS$ will be the same as the $k$-values along $S$.

In particular, $w\eta$ has $\delta$-shift zero, so it is a maximal weight with positive hub.  Hence by Proposition \ref{maxcont}, some component of the content of $w\eta$, say the $i$th component $\gamma_i$, is less than $a_i \leq a$. If we write $w\alpha=\sum_jt_j\alpha_j$, then the $i$th component of the content of $w\eta\pm aw\alpha$ is $\gamma_i\pm at_i$; but $w\eta\pm aw\alpha$ lies in $P(\Lm)\subseteq\Lm-Q_+$, so has content in which every component is non-negative.  Hence we must have $t_i=0$.  This means that the $i$th component of the content of every weight in the string $wS$ equals $\gamma_i<a_i$, so every weight in the string $wS$ is maximal.  So the $\delta$-shifts of all the weights in $wS$ are zero, and hence the $\delta$-shifts of all the weights in $S$ are zero.
\end{proof}
Now we can give some more precise information about the behavior of the $\delta$-shifts along an $\alpha$-string.

\begin{corl}
Suppose $\mathfrak g$ is of type $A^{(1)}_\ell$ or $D^{(2)}_{\ell+1}$.  Then along any $\alpha$-string, the $\delta$-shifts are either constant or strictly increasing to a symmetric central portion on which the $\delta$-shifts is fixed, after which they are strictly decreasing.
\end{corl}

\begin{proof}
For either of these types, all the $a_i$ equal $1$; so by Proposition \ref{internal}, if there is a string containing a maximal weight which is not an endpoint of that string, then the string consists entirely of maximal weights.

Take an $\alpha$-string, and write the weights in this string as
\[\lm,\lm+\alpha,\dots,\lm+t\alpha,\]
and write $s_i$ for the $\delta$-shift of $\lm+i\alpha$, for $0\leq i\leq t$.  Since the string must be symmetric with respect to the reflection $r_\alpha$ \cite[Proposition 11.1(a)]{Ka}, the $\delta$-shifts must be symmetrical too, i.e.\ $s_i=s_{t-i}$ for each $i$.  We prove the result by induction on $\max_is_i$.

Assume first that $s_i>0$ for each $i$.  Then we can shift the string by adding $\delta$ to each weight, and obtain a new $\alpha$-string in which the $\delta$-shifts are $s_0-1,\dots,s_t-1$; by induction the result holds for this new string, and so it holds for the original string.

So we may assume that some $s_i$ equals zero, i.e. there is a maximal weight in the string.  If there is a maximal weight which is not an endpoint of the string, then by Proposition \ref{internal} all the weights in the string are maximal; hence all the $k$-values equal zero, so the result holds.  Alternatively, suppose the only maximal weights in the string are the endpoints (note that if one endpoint is maximal, then so is the other, by the symmetry above).  Now shifting the string by adding $\delta$ to each weight and deleting the endpoints, we obtain a new string in which the $\delta$-shifts are $s_1-1,\dots,s_{t-1}-1$.  By induction the result holds for this new string, and so it holds for the original string.
\end{proof}

We remark that this result is certainly not true in other types.  In general, it is difficult to describe precisely the behaviour of $\delta$-shifts along a string, but they seem to vary approximately quadratically along the string.

\end{document}